\documentclass[11pt,reqno, oneside]{amsart}
\usepackage{pdfsync}  
\usepackage{geometry, tikz}
\usepackage{hyperref, fullpage}
\usepackage{diagbox}
\usepackage{subcaption, enumitem}
\usepackage{color}
\usepackage{amsmath}
\usepackage{multirow}
\usepackage{bm}
\usetikzlibrary{fit}
\usepackage{makecell}
\setcellgapes{2pt}

\newtheorem{thm}{Theorem}[section]
\newtheorem{cor}[thm]{Corollary}
\newtheorem{lem}[thm]{Lemma}
\newtheorem{prop}[thm]{Proposition}

\theoremstyle{definition}
\newtheorem{defn}[thm]{Definition}

\theoremstyle{definition}

\theoremstyle{definition}
\newtheorem{question}[thm]{Question}

\theoremstyle{definition}
\newtheorem{obs}[thm]{Observation}

\theoremstyle{definition}
\newtheorem{ex}[thm]{Example}

\newcommand{\D}{\mathcal{D}}
\newcommand{\T}{\mathcal{T}}
\newcommand{\M}{\mathcal{M}}

\renewcommand{\S}{\mathcal{S}}

\newcommand{\ol}{\overline}

\DeclareMathOperator{\Asc}{Asc}
\DeclareMathOperator{\Des}{Des}

\title{A new statistic on Dyck paths for counting 3-dimensional Catalan words}
\author{Kassie Archer}
\address[K. Archer]{University of Texas at Tyler, Tyler, TX 75799 USA}
\email{karcher@uttyler.edu}

\author{Christina Graves}
\address[C. Graves]{University of Texas at Tyler, Tyler, TX 75799 USA}
\email{cgraves@uttyler.edu}

\begin{document}

\maketitle

\begin{abstract}
A 3-dimensional Catalan word is a word on three letters so that the subword on any two letters is a Dyck path.  For a given Dyck path $D$, a recently defined statistic counts the number of Catalan words with the property that any subword on two letters is exactly $D$. In this paper, we enumerate Dyck paths with this statistic equal to certain values, including all primes. The formulas obtained are in terms of Motzkin numbers and Motzkin ballot numbers. 
\end{abstract}

\section{Introduction}

Dyck paths of semilength $n$ are paths from the origin $(0,0)$ to the point $(2n,0)$ that consist of steps $u=(1,1)$ and $d=(1,-1)$ and do not pass below the $x$-axis. Let us denote by $\D_n$ the set of Dyck paths of semilength $n$. It is a well-known fact that $\D_n$ is enumerated by the Catalan numbers. 

A \emph{3-dimensional Catalan path} (or just \emph{Catalan path}) is a higher-dimensional analog of a Dyck path. It is a path from $(0,0,0)$ to $(n,n,n)$ with steps $(1,0,0)$, $(0,1,0)$, and $(0,0,1)$, so at each lattice point $(x,y,z)$ along the path, we have $ x\geq y\geq z$. 
A \emph{3-dimensional Catalan word} (or just \emph{Catalan word}) is the word on the letters $\{x,y,z\}$ associated to a Catalan path where $x$ corresponds to the step in the $x$-direction $(1,0,0)$, $y$ corresponds to the step in the $y$-direction $(0,1,0)$, and $z$ corresponds to a step in the $z$ direction $(0,0,1)$. As an example, the complete list of Catalan words with $n=2$ is: 
$$xxyyzz \quad xxyzyz \quad xyxyzz \quad xyxzyz \quad xyzxyz.$$
Given a Catalan word $C$, the subword consisting only of $x$'s and $y$'s corresponds to a Dyck path by associating each $x$ to a $u$ and each $y$ to a $d$. Let us call this Dyck path $D_{xy}(C)$. Similarly, the subword consisting only of $y$'s and $z$'s is denoted by $D_{yz}(C)$ by relabeling each $y$ with a $u$ and each $z$ with a $d$.
For example, if $C=xxyxyzzxyyzz$, then $D_{xy}(C) = uudududd$ and $D_{yz}(C) =uudduudd$. 

Catalan words have been studied previously, see for example in \cite{GuProd20, Prod, Sulanke, Zeil}. 
In \cite{GuProd20} and \cite{Prod}, the authors study Catalan words $C$ of length $3n$  with $D_{xy}(C)=udud\ldots ud$ and determine that the number of such Catalan words is equal to $\frac{1}{2n+1}{{3n}\choose{n}}$. Notice that when $n=2$, the three Catalan words with this property are those in the above list whose $x$'s and $y$'s alternate.

In \cite{ArcGra21}, though it wasn't stated explicitly, it was found that the number of Catalan words $C$ of length $3n$ with $D_{xy}(C)=D_{yz}(C)$ is also $\frac{1}{2n+1}{{3n}\choose{n}}$. Such Catalan words have the property that the subword consisting of $x$'s and $y$'s is the same pattern as the subword consisting of $y$'s and $z$'s. For $n=2$, the three Catalan words with this property are:
\[ xxyyzz \quad xyxzyz \quad xyzxyz.\]
 The authors further show that for any fixed Dyck path $D$, the  number of Catalan words $C$ with $D_{xy}(C)=D_{yz}(C)=D$ is  given by
$$L(D) = \prod_{i=1}^{n-1} {r_i(D) + s_i(D) \choose r_i(D)}$$, where $r_i(D)$ is the number of down steps between the $i^{\text{th}}$ and $(i+1)^{\text{st}}$ up step in $D$, and $s_i(D)$ is the number of up steps between the $i^{\text{th}}$ and $(i+1)^{\text{st}}$ down step in $D$. The table in Figure~\ref{CatWord} shows all Dyck words $D \in \D_3$ and all corresponding Catalan paths $C$ with $D_{xy}(C)=D_{yz}(C)=D$.

\begin{figure}
\begin{center}
\begin{tabular}{c|c|l}
${D}$ & ${L(D)}$ & Catalan word $C$ with   ${D_{xy}(C)=D_{yz}(C)=D}$\\  \hline
$uuuddd$ & 1 & $xxxyyyzzz$\\ \hline
$uududd$ & 1 & $xxyxyzyzz$\\ \hline
$uuddud$ & 3 & $xxyyzzxyz, \ xxyyzxzyz,\  xxyyxzzyz$\\ \hline
$uduudd$ & 3 & $xyzxxyyzz, \ xyxzxyyzz, \ xyxxzyyzz$\\ \hline
$ududud$ & 4 & $xyzxyzxyz, \ xyzxyxzyz, \ xyxzyzxyz, \ xyxzyxzyz$
\end{tabular}
\end{center}
\caption{
All Dyck words $D \in \D_3$, and all corresponding Catalan words $C$ with ${D_{xy}(C)=D_{yz}(C)=D}$. There are $\frac{1}{7}{9 \choose 3} = 12$ total Catalan words $C$ of length $9$ with ${D_{xy}(C)=D_{yz}(C)}$. } \label{CatWord}
\end{figure}

As an application of the statistic $L(D)$, in \cite{ArcGra21} it was found that the number of 321-avoiding permutations of length $3n$ composed only of 3-cycles is equal to the following sum over Dyck paths:
\begin{equation}\label{eqnSumL2}
|\S_{3n}^\star(321)| = \sum_{D \in \D_n} L(D)\cdot 2^{h(D)},
\end{equation}
where $h(D)$ is the number of \emph{returns}, that is, the number of times a down step in the Dyck path $D$ touches the $x$-axis.

In this paper, we study this statistic more directly, asking the following question.
\begin{question} For a fixed $k$, how many Dyck paths $D \in \D_n$ have $L(D)=k$?\end{question} Equivalently, we could ask: how many Dyck paths $D \in \D_n$ correspond to exactly $k$ Catalan words $C$ with $D_{xy}(C) = D_{yz}(C) = D$? We completely answer this question when $k=1$, $k$ is a prime number, or $k=4$. The number of Dyck paths with $L=1$ is found to be the Motzkin numbers; see Theorem~\ref{TheoremL1}. When $k$ is prime, the number of Dyck paths with $L=k$ can be expressed in terms of the Motzkin numbers.  These results are found in Theorem~\ref{TheoremL2} and Theorem~\ref{TheoremLp}. Finally, when $k=4$, the number of Dyck paths with $L=4$ can also be expressed in terms of the Motzkin numbers; these results are found in Theorem~\ref{thm:L4}. A summary of these values for $k \in \{1,2,\ldots, 7\}$ can be found in the table in Figure~\ref{TableL}.

\begin{figure}[h]
\renewcommand{\arraystretch}{1.2}
\begin{tabular}{|r|l|c|c|}
\hline $|\D_n^k|$ & \textbf{Sequence starting at $n=k$} & \textbf{OEIS} & \textbf{Theorem} \\ \hline \hline
$|\D_n^1|$ & $1, 1, 2, 4, 9, 21, 51, 127, 323, \ldots$ & A001006 & Theorem \ref{TheoremL1}\\ \hline
$|\D_n^2|$ & $1,0,1,2,6,16,45,126,357,\ldots$ & A005717& Theorem \ref{TheoremL2}\\ \hline
$|\D_n^3|$ &$2, 2, 4, 10, 26, 70, 192, 534, \ldots$ & $2\cdot($A005773$)$ & \multirow{1}{*}{Theorem \ref{TheoremLp}} \\ \hline
$|\D_n^4|$ & $2, 5, 9, 25, 65, 181, 505, 1434, \ldots$ &$2\cdot($A025565$)$ + A352916 & Theorem \ref{thm:L4}\\ \hline
$|\D_n^5|$ &$2, 6, 14, 36, 96, 262, 726, 2034,  \ldots$ & $2\cdot($A225034$)$ &\multirow{1}{*}{Theorem \ref{TheoremLp}} \\ \hline
$|\D_n^6|$ & $14, 34, 92, 252, 710, 2026, 5844, \ldots$ && Section~\ref{SecRemarks}\\ \hline
$|\D_n^7|$ &$2, 10, 32, 94, 272, 784, 2260, 6524, \ldots$ & $2\cdot($A353133$)$ & \multirow{1}{*}{Theorem \ref{TheoremLp}}\\ \hline
\end{tabular}
\caption{The number of Dyck paths $D$ of semilength $n$ with $L(D)=k$.} \label{TableL}
\end{figure}

\section{Preliminaries}

We begin by stating a few basic definitions and introducing relevant notation.
\begin{defn} Let $D \in \D_n$.
\begin{enumerate}
\item An \emph{ascent} of $D$ is a maximal set of contiguous up steps; a \emph{descent} of $D$ is a maximal set of contiguous down steps.
\item If $D$ has $k$ ascents, the \emph{ascent sequence} of $D$ is given by $\Asc(D) = (a_1, a_2, \ldots, a_k)$ where $a_1$ is the length of the first ascent and $a_i - a_{i-1}$ is the length of the $i$th ascent for $2 \leq i \leq k$. 
\item Similarly, the \emph{descent sequence} of $D$ is given by $\Des(D) = (b_1, \ldots, b_k)$ where $b_1$ is the length of the first descent and $b_i - b_{i-1}$ is the length of the $i$th descent for $2 \leq i \leq k$. We also occasionally use the convention that $a_0=b_0 = 0$.
\item The \emph{$r$-$s$ array} of $D$ is the $2 \times n$ vector,
\[ \begin{pmatrix} r_1 & r_2 & \cdots  & r_{n-1}\\ s_1 & s_2 & \cdots & s_{n-1} \end{pmatrix} \]
where $r_i$ is the number of down steps between the $i^{\text{th}}$ and $(i+1)^{\text{st}}$ up step, and $s_i$ is the number of up steps between the $i^{\text{th}}$ and $(i+1)^{\text{st}}$ down step.
\item The statistic $L(D)$ is defined by $$L(D) = \prod_{i=1}^{n-1} {r_i(D) + s_i(D) \choose r_i(D)}.$$
\end{enumerate}
\end{defn}

We note that both the ascent sequence and the descent sequence are increasing, $a_i \geq b_i > 0$ for any $i$, and $a_k = b_k = n$ for any Dyck path with semilength $n$.  Furthermore, it is clear that any pair of sequences satisfying these properties produces a unique Dyck path.  There is also a relationship between the $r$-$s$ array of $D$ and the ascent and descent sequences as follows:

\begin{equation}\label{rs}
r_k = \begin{cases} 0 & \text{if } k \notin \Asc(D) \\ b_i - b_{i-1}& \text{if } k = a_i \text{ for some } a_i \in \Asc(D), \end{cases}
\end{equation}

\begin{equation}\label{rs2}
s_k = \begin{cases} 0 & \text{if } k \notin \Des(D) \\ a_{i+1} - a_i & \text{if } k = b_i \text{ for some } b_i \in \Des(D). \end{cases}
\end{equation}
The following example illustrates these definitions.

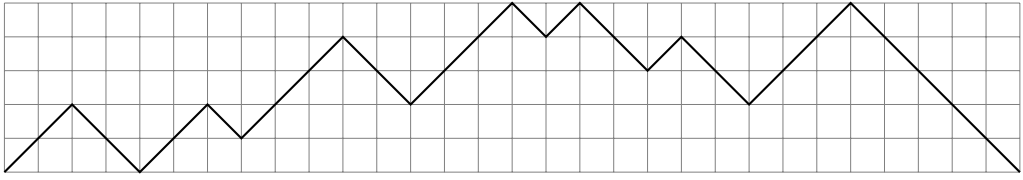
\begin{figure}
\begin{tikzpicture}[scale=.45]
\draw[help lines] (0,0) grid (30,5);
\draw[thick] (0,0)--(2,2)--(4,0)--(6,2)--(7,1)--(10,4)--(12,2)--(15,5)--(16,4)--(17,5)--(19,3)--(20,4)--(22,2)--(25,5)--(30,0);
\end{tikzpicture}
\caption{Dyck path $D$ with $L(D)=24$.}
\label{fig:dyckexample}
\end{figure}

\begin{ex} \label{RSEx} Consider the Dyck path
\[ D = uudduuduuudduuududdudduuuddddd, \]
which is pictured in Figure~\ref{fig:dyckexample}.  The ascent sequence and descent sequence of $D$ are
\[ \Asc(D) = (2, 4, 7, 10, 11, 12, 15) \quad\text { and } \quad \Des(D) = (2, 3, 5, 6, 8, 10, 15), \]
and the $r$-$s$ array of $D$ is
\[
\left( \begin{array}{cccccccccccccc}
0 & 2 & 0 & 1 & 0 & 0 & 2 & 0 & 0 & 1 & 2 & 2 & 0 & 0\\
0 & 2 & 3 & 0 & 3 & 1 & 0 & 1 & 0 & 3 & 0 & 0 & 0 & 0\end{array} \right).
\]
In order to compute $L(D)$, we note that if the $r$-$s$ array has at least one 0 in column $i$, then ${r_i + s_i \choose r_i} = 1$.  There are only two columns, columns 2 and 10, where both entries are nonzero.  Thus, \[ L(D) = {r_2 + s_2 \choose r_2}{r_{10} + s_{10} \choose r_{10}}={2 + 2 \choose 2} {1 + 3 \choose 3} = 24. \]
\end{ex}

The results in this paper rely on Motzkin numbers and Motzkin paths. A \emph{Motzkin path of length $n$} is a path from $(0,0)$ to $(n,0)$ composed of up steps $u=(1,1),$ down steps $d=(1,-1)$, and horizontal steps $h=(1,0)$, that does not pass below the $x$-axis. The set of Motzkin paths of length $n$ will be denoted $\mathcal{M}_n$ and the $n$th Motzkin number is  $M_n = |\mathcal{M}_n|$. (See OEIS A001006.)

We will also be considering modified Motzkin words as follows. Define $\mathcal{M}^*_n$ to be the set of words of length $n$ on the alphabet $\{h, u, d, *\}$ where the removal of all the $*$'s results in a Motzkin path. For each modified Motzkin word $M^* \in \M_{n-1}^*$, we can find a corresponding Dyck path in $\D_n$ by the procedure described in the following definition.

\begin{defn} \label{theta} Let  $M^* \in \mathcal{M}^*_{n-1}$. Define $D_{M^*}$ to be the Dyck path in $\D_n$ where $\Asc(D_{M^*})$ is the increasing sequence with elements from the set
\[ \{j : m_j = d \text{ or } m_j=*\} \cup \{n\} \]
and $\Des(D_{M^*})$ is the increasing sequence with elements from the set
\[ \{j : m_j = u \text{ or } m_j=*\} \cup \{n\}. \]
Furthermore, given $D\in\D_n$, define $M^*_D = m_1m_2\cdots m_{n-1} \in \mathcal{M}^*_{n-1}$ by 
\[ m_i = \begin{cases}  * & \text{if } r_i > 0 \text{ and } s_i > 0\\
u & \text{if } r_i=0 \text{ and } s_i>0\\
d & \text{if } r_i>0 \text{ and } s_i=0\\
h & \text{if } r_i=s_i=0.\\
\end{cases}
\]
\end{defn}

Notice that this process defines a one-to-one correspondence between $\mathcal{M}^*_{n-1}$ and $\D_n$. That is, $D_{M_D^*} = D$ and $M^*_{D_{M^*}} = M^*$.  Because this is used extensively in future proofs, we provide the following example.

\begin{ex} Let $D$ be the Dyck path defined in Example~\ref{RSEx}, pictured in Figure~\ref{fig:dyckexample}, with $r$-$s$ array:
\[
\left( \begin{array}{cccccccccccccc}
0 & 2 & 0 & 1 & 0 & 0 & 2 & 0 & 0 & 1 & 2 & 2 & 0 & 0\\
0 & 2 & 3 & 0 & 3 & 1 & 0 & 1 & 0 & 3 & 0 & 0 & 0 & 0\end{array} \right).
\]
The columns of the $r$-$s$ array help us to easily find $M^*_D$:
\begin{itemize} \item if column $i$ has two 0's, the $i$th letter in $M^*_D$ is $h$;
\item if column $i$ has a 0 on top and a nonzero number on bottom, the $i$th letter in $M^*_D$ is $u$;
\item if column $i$ has a 0 on bottom and a nonzero number on top, the $i$th letter in $M^*_D$ is $d$; and
\item if column $i$ has a two nonzero entries, the $i$th letter in $M^*_D$ is $*$.
\end{itemize}
Thus,
\[ M^*_D = h*uduuduh*ddhh. \]
Conversely, given $M^*_D$ as above, we find $D=D_{M_D^*}$ by first computing $\Asc(D)$ and $\Des(D)$. The sequence $\Asc(D)$ contains all the positions in $M^*_D$ that are either $d$ or $*$ while $\Des(D)$ contains all the positions in $M^*_D$ that are either $u$ or $*$. Thus,
\[ \Asc(D) = (2, 4, 7, 10, 11, 12, 15) \quad \text{and} \quad \Des(D) = (2, 3, 5, 6, 8, 10, 15).\]
\end{ex}

Notice that $L(D)$ is determined by the product of the binomial coefficients corresponding to the positions of $*$'s in $M^*_D$.  One final notation we use is to let  $\D_n^k$ be the set of Dyck paths $D$ with semilength $n$ and $L(D) = k$. With these definitions at hand, we are now ready to prove our main results.

\section{Dyck paths with $L=1$ or $L=\binom{r_k+s_k}{s_k}$ for some $k$} \label{SecRS}
In this section, we enumerate Dyck paths $D \in \D_n$ where $M^*_D$ has at most one $*$.  Because $L(D)$ is determined by the product of the binomial coefficients corresponding to the $*$ entries in $M^*_D$, Dyck paths with $L=1$ correspond exactly to the cases where $M^*_D$ has no $*$'s and are thus Motzkin paths.  Therefore, these Dyck paths will be enumerated by the well-studied Motzkin numbers.

\begin{thm} \label{TheoremL1}
For $n\geq 1$,  the number of Dyck paths $D$ with semilength $n$ and $L(D)=1$ is
\[  |\D_n^1| = M_{n-1}, \] where $M_{n-1}$ is the $(n-1)^{\text{st}}$ Motzkin number. 
\end{thm}

\begin{proof} Let $D \in \D_n^1$. Since $L(D) = 1$, it must be the case that either $r_i(D) = 0$ or $s_i(D) = 0$ for all $i$.  By Definition~\ref{theta}, $M^*_D$ consists only of elements in $\{h, u, d\}$ and is thus a Motzkin path in $\mathcal{M}_{n-1}$. This process is invertible, as given any Motzkin path $M \in \mathcal{M}_{n-1} \subseteq \mathcal{M}^*_{n-1}$, we have $D_{M_D} = D$.
\end{proof}

As an example, the table in Figure \ref{L1Figure} shows the $M_4 = 9$ Dyck paths in $\D_5^1$ and their corresponding Motzkin paths.  

\begin{figure} 
\begin{center}
\begin{tabular}{c|c|c|c}
Dyck path $D$& $r$-$s$ array & $M^*_D$ & Motzkin path\\ \hline 
\begin{tikzpicture}[scale=.2, baseline=0]
\draw[help lines] (0,0) grid (10,5);
\draw[thick] (0,0)--(5,5)--(10,0);
\node at (0,5.2)  {\color{red!90!black}\ };
\end{tikzpicture} & $\begin{pmatrix} 0 & 0 & 0 & 0\\0&0&0&0\end{pmatrix}$  & $hhhh$ & 
\begin{tikzpicture}[scale=.3]
\draw[help lines] (0,0) grid (4,1);
\draw[thick] (0,0)--(4,0);
\end{tikzpicture}
\\ \hline
\begin{tikzpicture}[scale=.2]
\draw[help lines] (0,0) grid (10,4);
\draw[thick] (0,0)--(4,4)--(5,3)--(6,4)--(10,0);
\node at (0,4.2)  {\color{red!90!black}\ };
\end{tikzpicture} & \begin{tabular}{c}$\begin{pmatrix} 0 & 0 & 0 & 1\\1&0&0&0\end{pmatrix}$\end{tabular} & $uhhd$ & 
\begin{tikzpicture}[scale=.3]
\draw[help lines] (0,0) grid (4,1);
\draw[thick] (0,0)--(1,1)--(2,1)--(3,1)--(4,0);
\end{tikzpicture}\\  \hline
\begin{tikzpicture}[scale=.2]
\draw[help lines] (0,0) grid (10,4);
\draw[thick] (0,0)--(4,4)--(6,2)--(7,3)--(10,0);
\node at (0,4.5)  {\color{red!90!black}\ };
\end{tikzpicture} & $\begin{pmatrix} 0 & 0 & 0 & 2\\0&1&0&0\end{pmatrix}$ & $huhd$ & 
\begin{tikzpicture}[scale=.3]
\draw[help lines] (0,0) grid (4,1);
\draw[thick] (0,0)--(1,0)--(2,1)--(3,1)--(4,0);
\end{tikzpicture}\\ \hline
\begin{tikzpicture}[scale=.2]
\draw[help lines] (0,0) grid (10,4);
\draw[thick] (0,0)--(4,4)--(7,1)--(8,2)--(10,0);
\node at (0,4.5)  {\color{red!90!black}\ };
\end{tikzpicture} & $\begin{pmatrix} 0 & 0 & 0 & 3\\0&0&1&0\end{pmatrix}$ & $hhud$ &
\begin{tikzpicture}[scale=.3]
\draw[help lines] (0,0) grid (4,1);
\draw[thick] (0,0)--(1,0)--(2,0)--(3,1)--(4,0);
\end{tikzpicture}\\ \hline
\begin{tikzpicture}[scale=.2]
\draw[help lines] (0,0) grid (10,4);
\draw[thick] (0,0)--(3,3)--(4,2)--(6,4)--(10,0);
\node at (0,4.5)  {\color{red!90!black}\ };
\end{tikzpicture} & $\begin{pmatrix} 0 & 0 & 1 & 0\\2&0&0&0\end{pmatrix}$ & $uhdh$ &
\begin{tikzpicture}[scale=.3]
\draw[help lines] (0,0) grid (4,1);
\draw[thick] (0,0)--(1,1)--(2,1)--(3,0)--(4,0);
\end{tikzpicture}\\ \hline
\begin{tikzpicture}[scale=.2]
\draw[help lines] (0,0) grid (10,3);
\draw[thick] (0,0)--(3,3)--(5,1)--(7,3)--(10,0);
\node at (0,3.5)  {\color{red!90!black}\ };
\end{tikzpicture} & $\begin{pmatrix} 0 & 0 & 2 & 0\\0&2&0&0\end{pmatrix}$ & $hudh$ &
\begin{tikzpicture}[scale=.3]
\draw[help lines] (0,0) grid (4,1);
\draw[thick] (0,0)--(1,0)--(2,1)--(3,0)--(4,0);
\end{tikzpicture}\\ \hline
\begin{tikzpicture}[scale=.2]
\draw[help lines] (0,0) grid (10,4);
\draw[thick] (0,0)--(2,2)--(3,1)--(6,4)--(10,0);
\node at (0,4.5)  {\color{red!90!black}\ };
\end{tikzpicture} & $\begin{pmatrix} 0 & 1 & 0 & 0\\3&0&0&0\end{pmatrix}$ & $udhh$ &
\begin{tikzpicture}[scale=.3]
\draw[help lines] (0,0) grid (4,1);
\draw[thick] (0,0)--(1,1)--(2,0)--(3,0)--(4,0);
\end{tikzpicture}\\ \hline
\begin{tikzpicture}[scale=.2]
\draw[help lines] (0,0) grid (10,3);
\draw[thick] (0,0)--(3,3)--(4,2)--(5,3)--(6,2)--(7,3)--(10,0);
\node at (0,3.5)  {\color{red!90!black}\ };
\end{tikzpicture} & $\begin{pmatrix} 0 & 0 & 1 & 1\\1&1&0&0\end{pmatrix}$ & $uudd$ &
\begin{tikzpicture}[scale=.3]
\draw[help lines] (0,0) grid (4,2);
\draw[thick] (0,0)--(1,1)--(2,2)--(3,1)--(4,0);
\end{tikzpicture}\\ \hline
\begin{tikzpicture}[scale=.2]
\draw[help lines] (0,0) grid (10,3);
\draw[thick] (0,0)--(2,2)--(3,1)--(5,3)--(7,1)--(8,2)--(10,0);
\node at (0,3.5)  {\color{red!90!black}\ };
\end{tikzpicture} & $\begin{pmatrix} 0 & 1 & 0 & 2\\2&0&1&0\end{pmatrix}$ & $udud$ &
\begin{tikzpicture}[scale=.3]
\draw[help lines] (0,0) grid (4,1);
\draw[thick] (0,0)--(1,1)--(2,0)--(3,1)--(4,0);
\end{tikzpicture}\\ \hline
\end{tabular}
\end{center}
\caption{The nine Dyck paths of semilength 5 having $L=1$  and their corresponding Motzkin paths of length 4.} \label{L1Figure}
\end{figure}

 We now consider Dyck paths $D \in \D_n$ where $D_{M^*}$ has exactly one $*$. Such Dyck paths have  $L=\binom{r_k+s_k}{s_k}$ where $k$ is the position of $*$ in $D_{M^*}$. We call the set of Dyck paths of semilength $n$ with $L=\binom{r+s}{s}$ obtained in this way $\D_{n}^{r,s}$. 
 
For ease of notation, if $D \in \D_{n}^{r,s}$, define
\begin{itemize}
\item $x(D)$ to be the number of ups before the $*$ in $M^*_D$, and
\item $y(D)$ be the number of downs before the $*$ in $M^*_D$.
\end{itemize}
We can then easily compute the value of $L(D)$ based on $x(D)$ and $y(D)$ as stated in the following observation.

\begin{obs}\label{obsRS} Suppose $D \in \D_{n}^{r,s}$ and write $x=x(D)$ and $y=y(D)$. Then in $M^*_D$, the following are true.
\begin{itemize}
\item The difference in positions of the $(y+1)$st occurrence of either $u$ or $*$ and the $y$th occurrence of $u$ is $r$; or, when $y=0$, the first occurrence of $u$ is in position $r$.
\item The difference in positions of the $(x+2)$nd occurrence of either $d$ or $*$ and the $(x+1)$st occurrence of either $d$ or $*$ is $s$; or, when $x$ is the number of downs in $M^*_D$, the last occurrence of $d$ is in position $n-s$.
\end{itemize}
\end{obs}

\begin{ex} Consider the Dyck path
\[ D = uuuuudduudddduuduudddd. \]
The ascent sequence and descent sequence of $D$ are
\[ \Asc(D) = (5, 7, 9, 11) \quad\text { and } \quad \Des(D) = (2, 6, 7, 11), \]
and the $r$-$s$ array of $D$ is
\[
\left( \begin{array}{cccccccccc}
0 & 0 & 0 & 0 & 1 & 0 & 4 & 0 & 1 & 0 \\
0 & 2 & 0 & 0 & 0 & 2 & 2 & 0 & 0 & 0 \end{array} \right).
\]
There is only one column, column 7, where both entries are nonzero.  Thus, \[ L(D) = {r_7 + s_7 \choose r_7}={4 + 2 \choose 4} = 15, \] and $D \in \D_{11}^{4,2}$. Note also that \[ M^*_D = huhhdu*hdh \] has exactly one $*$. Now let's compute $L(D)$ more directly using Observation~\ref{obsRS}. Notice $x(D) = 2$ and $y(D) = 1$ since there are two $u$'s before the $*$ in $M^*_D$ and one $d$ before the $*$. In this case, the position of the second occurrence of either $u$ or $*$ is 6 and the position of the first occurrence of $u$ is 2, so $r=6-2=4$. Since there are only two downs in $M^*_D$, we note the last $d$ occurs in position 9, so $s=11-9=2$.
\end{ex}
 
 In order to proceed, we need to define the  Motzkin ballot numbers.  The \emph{Motzkin ballot numbers} are the number of Motzkin paths that have their first down step in a fixed position. These numbers appear in \cite{Aigner98} and are similar to the well-known Catalan ballot numbers (see \cite{Brualdi}).  If $n \geq k$, we let $\mathcal{T}_{n,k}$ be the set of Motzkin paths of length $n$ with the first down in position $k$, and we define $\T_{k-1, k}$ to be the set containing the single Motzkin path consisting of $k-1$ horizontal steps.
 
Given any Motzkin path $M$, define the \emph{reverse of $M$}, denoted $M^R$, to be the Motzkin path found be reading $M$ in reverse and switching $u$'s and $d$'s. For example, if $M=huuhdhd$, $M^R = uhuhddh$.  Given $M \in \mathcal{T}_{n,k}$, the Motzkin path $M^R$ has its last up in position $n-k+1$. 
 
The following lemma gives the generating function for the Motzkin ballot numbers $T_{n,k} = |\mathcal{T}_{n,k}|$.

\begin{lem} \label{lemGFt}
For positive integers $n \geq k$, let $T_{n,k} = |\T_{n,k}|$.  Then for a fixed $k$, the generating function for $T_{n,k}$ is given by
\[  \sum_{n=k-1}^{\infty} T_{n,k}x^n =  \left(1+xm(x)\right)^{k-1}x^{k-1}. \]
\end{lem}

\begin{proof}
Consider a Motzkin path of length $n$ with the first down in position $k$.  It can be rewritten as 
\[ a_1a_2\cdots a_{k-1} \alpha_1 \alpha_2 \cdots \alpha_{k-1} \]
where either
\begin{itemize}
\item $a_i = f$ and $\alpha_i$ is the empty word, or
\item $a_i = u$ and $\alpha_i$ is $dM_i$ for some Motzkin word $M_i$,
\end{itemize}
for any $1 \leq i \leq k-1$.
The generating function is therefore $(x + x^2m(x))^{k-1}$.
\end{proof}

In later proofs we decompose certain Motzkin paths as shown in the following definition.

\begin{defn} \label{PrPs} Let $r$, $s$, and $n$ be positive integers with $n \geq r+ s -2$, and let $P \in \mathcal{T}_{n, r+s-1}$. Define $P_s$ to be the maximal Motzkin subpath in $P$ that begins at the $r$th entry, and define $P_r$ be the Motzkin path formed by removing $P_s$ from $P$.
\end{defn}

Given $P \in \mathcal{T}_{n, r+s-1}$, notice that $P_r \in \mathcal{T}_{\ell, r}$ for some $r-1 \leq \ell \leq n-s + 1$ and $P_s \in \mathcal{T}_{n-\ell, s}$. In other words, the first down in $P_s$ must be in position $s$ (or $P_s$ consists of $s-1$ horizontal steps), and the first down in $P_r$ must be in position $r$  (or $P_r$ consists of $r-1$ horizontal steps). This process is invertible as follows.  Given $P_r \in \mathcal{T}_{\ell,r}$ and $P_s \in \mathcal{T}_{n-\ell,s}$, form a Motzkin path $P \in \mathcal{T}_{n, r+s-1}$ by inserting $P_s$ after the $(r-1)$st element in $P_r$.  

Because this process is used extensively in subsequent proofs, we illustrate this process with an example below.

\begin{ex} \label{exBreakM} Let $r=3$, $s=4$, and $n=13$.  Suppose $P = uhuhhdhhdudud \in \mathcal{T}_{13,6}$.  By definition, $P_s$ is the maximal Motzkin path obtained from $P$ by starting at the 3rd entry: 
\[ P = uh\framebox{$uhhdhh$}dudud. \]
Thus, $P_s = uhhdhh \in \mathcal{T}_{6, 4}$ as seen in the boxed subword of $P$  above, and $P_r =  uhdudud \in \mathcal{T}_{7, 3}$. 
Conversely, given $P$ as shown above and $r=3$, we note that the maximal Motzkin path in $P_s$ starting at position 3 is exactly the boxed part $P_s$.
\end{ex}

Using the Motzkin ballot numbers and this decomposition of Motzkin paths, we can enumerate the set of Dyck paths in $\mathcal{D}_n^{r,s}$. These are enumerated by first considering the number of returns. Suppose a Dyck path $D \in \D_n$ has a return after $2k$ steps with $k < n$. Then $r_k(D)$ is the length of the ascent starting in position $2k+1$, and $s_k(D)$ is the length of the descent ending where $D$ has a return. Thus, the binomial coefficient ${r_k+ s_k \choose r_k} > 1$. This implies that if $D \in  \mathcal{D}_n^{r,s}$, it can have at most two returns (including the end). Dyck paths in $\mathcal{D}_n^{r,s}$ that have exactly two returns are counted in Lemma~\ref{RSHit2}, and those that have a return only at the end are counted in Lemma~\ref{RSHit1}.

\begin{lem}\label{RSHit2} For $r\geq 1, s\geq 1$, and $n\geq r+s$,  the number of Dyck paths $D \in \D_n^{r,s}$ that have two returns is $T_{n-2, r+s-1}$.
\end{lem}

\begin{proof} We will find a bijection between the set of Dyck paths in $\D_n^{r,s}$ that have exactly two returns and $\mathcal{T}_{n-2, r+s-1}$. First, suppose $P \in \mathcal{T}_{n-2, r+s-1}$. Thus, there is some $r-1 \leq \ell \leq n-s+1$ so that $P_r \in \mathcal{T}_{\ell, r}$ and $P_s \in \mathcal{T}_{n-2-\ell, s}$ where $P_r$ and $P_s$ are as defined in Definition~\ref{PrPs}. 

Now create the modified Motzkin word $M^* \in \M_{n-1}^*$ by concatenating the reverse of $P_r$, the letter $*$, and the word $P_s$; that is, $M^* = P_r^R*P_s$. Because $P_r$ and $P_s$ have a combined total length of $n-2$, the modified Motzkin word $M^*$ is length $n-1$. Let $D = D_{M^*}$ as defined in Definition~\ref{theta} and let $x = x(D)$ and $y= y(D)$. Since $M^*$ has only the Motzkin word $P_r^R$ before $*$, we have $x=y$ and $D$ must have exactly two returns.

Using Observation~\ref{obsRS}, we can show that $D \in \D_n^{r,s}$ as follows. The $(y+1)$st occurrence of either a $u$ or $*$ is the $*$ and the $y$th occurrence of $u$ is the last $u$ in $P_r^R$; the difference in these positions is $r$. Also, the $(x+1)$st occurrence of either a $d$ or $*$ is the $*$ and the $(x+2)$nd occurrence of either a $d$ or $*$ is the first $d$ in $P_s$; the difference in these positions is $s$.

To see that this process is invertible, consider any Dyck path $D\in\D_n^{r,s}$ that has exactly two returns. Since $D\in\D_n^{r,s}$, $M^*_D$ has exactly one $*$. Furthermore, since $D$ has a return after $2k$ steps for some $k < n$, it must be that $*$ decomposes $M^*_D$ into two Motzkin paths.  That is, the subword of $M^*_D$ before the $*$ is a Motzkin path as well as the subword of $M^*_D$ after the $*$.  We will call the subword of $M^*_D$ consisting of the first $k-1$ entries $M_r$ and the subword of $M^*_D$ consisting of the last $n-1-k$ entries $M_s$.

Since $r_k=r$ and there are the same number of ups and downs before the $*$ in $M^*_D$, the last up before $*$ must be in position $k-r$.  Similarly, since $s_k=s$, the first down after $*$ must be in position $k+s$. Thus, $M_r^R \in \T_{k-1,r}$ and $M_s \in \T_{n-1-k, s}$. Let $P$ be the Motzkin path formed by inserting $M_s$ after the $(r-1)$st element in $M_r^R$. Then $P \in \T_{n-2, r+s-1}$ as desired. \end{proof}

The following example shows the correspondence.

\begin{ex} Let $r=3$, $s=4$, and $n=15$.  Suppose $P = uhuhhdhhdudud \in \mathcal{T}_{13,6}$. The corresponding Dyck path $D \in \D_{15}^{3, 4}$ is found as follows. First, find $P_r =  uhdudud$ and $P_s = uhhdhh$ as in Example~\ref{exBreakM}. Then let $M^* = P_r^R*P_s$ or
\[ M^* = ududuhd*uhhdhh.\]
Letting $D = D_{M^*}$, we see that $x(D) = y(D) = 3$.  The fourth occurrence of either $u$ or $*$ is the $*$ in position $8$, and the third occurrence of $u$ is in position $5$, so $r=8-5=3$. Similarly, the fourth occurrence of either $d$ or $*$ is the $*$ in position 8, and the fifth occurrence of $d$ is in position 12, so $s=12-8=4$ as desired.

\sloppypar{For completion, we write the actual Dyck path $D$ using Definition~\ref{theta} by first seeing $\Asc(D)~=~(2, 4, 7, 8, 12,15)$ and $\Des(D) = (1, 3, 5, 8, 9, 15)$. Thus}
\[ D = uuduudduuudduddduuuuduuudddddd.\]
\end{ex}

Lemma~\ref{RSHit2} counted the Dyck paths in $\D_n^{r,s}$ that have exactly two returns;  the ensuing lemma counts those Dyck paths in $\D_n^{r,s}$ that have only one return (at the end).

\begin{lem} \label{RSHit1} For $r\geq 1, s\geq 1$, and $n\geq r+s+2$,  the number of Dyck paths $D \in \D_n^{r,s}$ that only have a return at the end is \[ \sum_{i=0}^{n-2-s-r}(i+1)M_i T_{n-4-i, r+s-1}. \]
\end{lem}

\begin{proof}
Consider a pair of Motzkin paths, $M$ and $P$, where $M$ is length $i$ with $0 \leq i \leq n-2-s-r$, and $P \in \mathcal{T}_{n-4-i, r+s-1}$. For each such pair, we consider $1 \leq j \leq i+1$ and find a corresponding Dyck path $D\in\D_n^{r,s}$. Thus, there will be $i+1$ corresponding Dyck paths for each pair $M$ and $P$.  Each Dyck path $D$ will have exactly one $*$ in $M^*_D$.

We begin by letting $\ol{M}^*$ be the modified Motzkin path obtained by inserting $*$ before the $j$th entry in $M$ or at the end if $j=i+1$. Let  $\ol{x}$ be the number of ups before the $*$ in  $\ol{M}^*$, and let $\ol{y}$ be the  number of downs before the $*$ in  $\ol{M}^*$.  

Recall that by Definition~\ref{PrPs}, there is some $r-1 \leq \ell \leq n-3-s-i$ so that $P$ can be decomposed into $P_r \in \mathcal{T}_{\ell, r}$ and $P_s \in \mathcal{T}_{n-4-i-\ell, s}$. We now create a modified Motzkin word, $M^* \in \M^*_{n-1}$ by inserting one $u$, one $d$, $P_r^R$, and $P_s$ into $\ol{M}^*$ as follows. 
\begin{enumerate}
\item Insert a $d$ followed by $P_s$ immediately before the  $(\ol{x}+1)$st $d$ in $\ol{M}^*$ or at the end if $\ol{x}$ is equal to the number of downs in $\ol{M}^*$.
\item Insert the reverse of $P_r$ followed by $u$ after the $\ol{y}$th $u$ or at the beginning if $\ol{y}=0$.
\end{enumerate}
Call the resulting path $M^*$. We claim that $D_{M^*}\in \mathcal{D}_n^{r,s}$ and that $D_{M^*}$ only has one return at the end. For ease of notation, let $D = D_{M^*}, x=x(D)$, and  $y=y(D)$.  Notice that the number of downs (and thus the number of ups) in $P_r$ is $y-\ol{y}$. Then the $(y+1)$st $u$ or $*$ in $M^*$ is the inserted $u$ following $P_r^R$ from Step (2), and the $y$th $u$ is the last $u$ in $P_r^R$. The difference in these positions is $r$. Similarly, the $(x+1)$st $d$ or $*$ in $M^*$ is the inserted $d$ before the $P_s$ from Step (1), and the $(x+2)$nd $d$ or $*$ in $M^*$ is the first down in $P_s$.  The difference in these positions is $s$, and thus by Observation~\ref{obsRS}, $D \in \mathcal{D}_n^{r,s}$.

To see that $D$ only has one return at the end, we note that  the only other possible place $D$  can  have a return is after $2k$ steps where $k = \ell + j + 1$, the position of $*$ in $M^*$.  However, $x > y$ so $D$  only has one return at the end.

We now show that this process is invertible. Consider any Dyck path $D\in\D_n^{r,s}$ that has one return at the end. Since $D$ only has one return at the end, the $*$ does not decompose $M^*_D$ into two Motzkin paths, and we must have $x(D)>y(D)$.

Let $P_1$ be the maximal Motzkin word immediately following the $(x+1)$st occurrence of $d$ or $*$ in $M^*_D$. Note that $P_1$ must have its first down in position $s$ or $P_1$ consists of $s-1$ horizontal steps. Let $P_2$ be the maximal Motzkin word preceding the $(y+1)$st up in $M^*$. Then either $P_2$ consists of $r-1$ horizontal step or the last $u$ in $P_2$ is $r$ from the end; that is, the first $d$ in $P_2^R$ is in position $r$.

Since $x>y$, the $(y+1)$st $u$ comes before the $x$th $d$. Thus, deleting the $*$, the $(y+1)$st $u$, the $x$th $d$, $P_1$, and $P_2$ results in a Motzkin path we call $M$. Note that if $M$ is length $i$, then the combined lengths of $P_1$ and $P_2$ is length $n-4-i$. This inverts the process by letting $P_s=P_1$ and $P_r=P_2^R.$  \end{proof}

 We again illustrate the correspondence from the above proof with an example.

 \begin{ex}
 Let $r=3$, $s=4$, $n=24$, and consider the following pair of Motzkin paths
 \[ M = uudhudd \quad \text{ and } \quad P = uhuhhdhhdudud. \]
 As in Example~\ref{exBreakM}, $P_r =  uhdudud$ and $P_s = uhhdhh$. Following the notation in the proof of Lemma~\ref{RSHit1}, we have $i = 7$. Our goal is to find $8$ corresponding Dyck paths for each $1 \leq j \leq 8$. If $j = 1$, we first create $\ol{M}^*$ by inserting $*$ before the 1st entry in M:
 \[ \ol{M}^* = *uudhudd.\]
Now there are $\ol{x} = 0$ ups and $\ol{y}=0$ downs before the $*$ in $\ol{M}^*$. Thus, we form $M^*$ by inserting $P^R_ru$ at the beginning of $\ol{M}^*$ and $dP_s$ immediately before the $1$st down in $\ol{M}^*$ yielding
 \[ M^*= \framebox{$ududuhd$}\ \bm{u}* uu\bm{d}\ \framebox{$uhhdhh$}\ dhudd. \]
 The paths $P_r^R$ and $P_s$ are boxed in the above notation and the inserted $u$ and $d$ are in bold.
 
 If $D=D_{M^*}$, then $x(D) = 4$ and $y(D) = 3$ because there are four $u$'s and three $d$'s before $*$ in $M^*$. The $(y+1)$st (or fourth) occurrence of $u$ or $*$ in $M^*$ is the bolded $u$ in position 8, and the third occurrence of $u$ is the last $u$ in $P_r^R$ in position 5;  thus $r=3$. Similarly, the $(x+2)$nd (or sixth) occurrence of $d$ or $*$ is the first $d$ in $P_s$ in position 16, and the fifth occurrence of $d$ or $*$ is the bolded $d$ in position 12 giving us $s=4$. It is clear that $D$ only has one return since $x > y$.
 
 This process can be followed in the same manner for $2 \leq j \leq 8$ to find all $8$ corresponding Dyck paths for the pair $M$ and $P$.  The table in Figure~\ref{RSEx2} shows these paths.
\end{ex}

 \begin{figure}
 \begin{center}
 {\renewcommand{\arraystretch}{2}
 \begin{tabular}{c|c|c|c|c}
 $j$ & $\ol{M}^*$ & $\ol{x}$ & $\ol{y}$ & $M^*$ \\ \hline
 1 & $*uudhudd$ & 0 & 0 & $ \framebox{$ududuhd$}\ \bm{u}* uu\bm{d}\ \framebox{$uhhdhh$}\ dhudd$\\ \hline
  2 & $u*udhudd$ & 1 & 0 & $ \framebox{$ududuhd$}\ \bm{u} u*udhu\bm{d}\ \framebox{$uhhdhh$}\ dd$\\ \hline
   3 & $uu*dhudd$ & 2 & 0 & $ \framebox{$ududuhd$}\ \bm{u} uu*dhud\bm{d}\ \framebox{$uhhdhh$}\ d$\\ \hline
    4 & $uud*hudd$ & 2 & 1 & $u \framebox{$ududuhd$}\ \bm{u}ud*hud\bm{d}\ \framebox{$uhhdhh$}\ d$\\ \hline
     5 & $uudh*udd$ & 2 & 1 & $u\framebox{$ududuhd$}\ \bm{u}udh*ud\bm{d}\ \framebox{$uhhdhh$}\ d$\\ \hline
      6 & $uudhu*dd$ & 3 & 1 & $u\framebox{$ududuhd$}\ \bm{u}udhu*dd\bm{d}\ \framebox{$uhhdhh$}\ $\\ \hline
       7 & $uudhud*d$ & 3 & 2 & $uu \framebox{$ududuhd$}\ \bm{u}dhud*d\bm{d}\ \framebox{$uhhdhh$}\ $\\ \hline
        8 & $uudhudd*$ & 3 & 3 & $uudhu\framebox{$ududuhd$}\ \bm{u}dd*\bm{d}\ \framebox{$uhhdhh$}\ $\\ \hline
 \end{tabular}}
 \end{center}
\caption{Given $r=3,$ $s=4,$ $n=24$, and the pair of Motzkin paths $M~=~uudhudd \in \M_7$ and $P = uhuhhdhhdudud \in \T_{13, 6}$, the Dyck words formed by $D_{M^*}$ are the 8 corresponding Dyck paths in $\D_{24}^{3,4}$ that only have one return.} \label{RSEx2}
 \end{figure}

By combining Lemmas~\ref{RSHit2} and \ref{RSHit1}, we have the following proposition which enumerates $\D_n^{r,s}$.
 
 \begin{prop} \label{oneterm} For $r\geq 1, s\geq 1$, and $n\geq r+s$,  the number of Dyck paths $D \in \D_n^{r,s}$ is 
\[ |\D_n^{r,s}| =T_{n-2,r+s-1} + \sum_{i=0}^{n-2-s-r}(i+1)M_i T_{n-4-i, r+s-1}.\]
\end{prop}

\begin{proof}
Dyck paths in $\mathcal{D}_n^{r,s}$ can have at most two returns. Thus, this is a direct consequence of Lemmas ~\ref{RSHit2} and \ref{RSHit1}. \end{proof}

Interestingly, we remark that the formula for $|\D_n^{r,s}|$ only depends on the sum $r+s$ and not the individual values of $r$ and $s$. For example, $|\D_n^{1,3}| = |\D_n^{2,2}|$. Also, because the formula for $|\D_n^{r,s}|$ is given in terms of Motzkin paths, we can easily extract the generating function for these numbers using Lemma~~\ref{lemGFt}.

\begin{cor} For $r, s \geq 1$, the generating function for $|\D_n^{r,s}|$ is
\[ x^{r+s}(1+xm(x))^{r+s-2}\left(1 + x^2(xm(x))' \right). \]
\end{cor}

\section{Dyck paths with $L=p$ for prime $p$}

When $L=p$, for some prime $p$, we must have that every term in the product $\prod_{i=1}^{n-1} {r_i + s_i \choose r_i}$ is equal to 1 except for one term which must equal $p$.  In particular, we must have that there is exactly one $1\leq k\leq n-1$ with $r_k\neq 0$ and $s_k\neq 0$. Furthermore, we must have that either $r_k=1$ and $s_k=p-1$ or $r_k=p-1$ and $s_k=1$.  Therefore, when $L =2$, we have \[ |\mathcal{D}_n^2| = |\mathcal{D}_n^{1,1}|. \]  When $L=p$ for an odd prime number, we have \[ |\mathcal{D}_n^p| = |\mathcal{D}_n^{1,p-1}| +  |\mathcal{D}_n^{p-1,1}| = 2|\mathcal{D}_n^{1,p-1}|. \]  Thus the results from the previous section can be used in the subsequent proofs.

\begin{thm} \label{TheoremL2} For $n\geq 4$,  the number of Dyck paths with semilength $n$ and $L=2$ is
\[ |\D_n^2| = (n-3)M_{n-4}, \]where $M_{n-4}$ is the $(n-4)$th Motzkin number. Additionally, $|\D_2^2| =1$ and $|\D_3^2| = 0.$  Thus the generating function for $|\D_n^2|$ is given by
\[ L_2(x) = x^2 + x^4\left(xm(x)\right)' \]
where $m(x)$ is the generating function for the Motzkin numbers.
\end{thm}

\begin{proof} By Proposition~\ref{oneterm}, for $n \geq 2$, \[ |\D_n^{1,1}| =T_{n-2,1} + \sum_{i=0}^{n-4}(i+1)M_i T_{n-4-i, 1}.\]
In the case where $n=3$ or $n=4$, the summation is empty and thus $|\D_2^2| = T_{0,1} = 1$ and $|\D_3^2| = T_{1,1} = 0$.  For $n \geq 4$, the term $T_{n-2,1} = 0$.  Furthermore, the terms in the summation are all 0 except when $i=n-4$.  Thus, \[ |\D_n^{1,1}| = (n-3)M_{n-4}T_{0,1} \]
or \[ |\D_n^2| = (n-3)M_{n-4}. \] \end{proof}

The sequence for the number of Dyck paths of semilength $n$ with $L=2$ is given by:
\[  |\D_n^2| = 1,0,1,2,6,16,45,126,357,\ldots \]
This can be found at OEIS A005717.

Because the formula for $|\D_n^2|$ is much simpler than the one found in Proposition~\ref{oneterm}, the correspondence between Dyck paths in $\D_n^2$ and Motzkin paths of length $n-4$ is actually fairly straightforward.  For each Motzkin word of length $n-4$, there are $n-3$ corresponding Dyck paths of semilength $n$ having $L=2$.  These corresponding Dyck paths are found by modifying the original Motzkin word $n-3$ different ways.  Each modification involves adding a $u$, $d$, and placeholder $*$ to the original Motzkin word. The $n-3$ distinct modifications correspond to the $n-3$ possible positions of the placeholder $*$ into the original Motzkin word.  

As an example, in the case where $n=6$, there are $M_2 = 2$ Motzkin words of length 2.  For each word, we can insert a placeholder $*$ in $n-3=3$ different positions, and thus there are a total of 6 corresponding Dyck paths of semilength 6 having $L=2$.  Figure~\ref{L2Figure} provides the detailed process when $n=6$.

\begin{figure}[t]
\begin{center}
\begin{tabular}{c|c|c|c|c}
$\begin{matrix} \text{Motzkin}\\ \text{word} \end{matrix}$ & $M^*$ & $\begin{matrix} \Asc(D)\\ \Des(D) \end{matrix}$ & $r$-$s$ array & Dyck path, $D$\\ \hline 
 &  \vspace*{-.1cm} $\bm{u}*hh\bm{d}$ & $\begin{matrix} (2&5&6)\\(1&2&6)\end{matrix}$ & $\begin{pmatrix}0&1&0&0&1\\3&1&0&0&0\end{pmatrix}$&
\begin{tikzpicture}[scale=.2]
\draw[help lines] (0,0) grid (12,4);
\draw[thick] (0,0)--(2,2)--(3,1)--(6,4)--(7,3)--(8,4)--(12,0);
\node at (0,5)  {\color{red!90!black}\ };
\end{tikzpicture}\\ 
 $hh$ &$\bm{u}h*h\bm{d}$ & $\begin{matrix} (3&5&6)\\(1&3&6) \end{matrix}$ & $\begin{pmatrix}0&0&1&0&2\\2&0&1&0&0\end{pmatrix}$&
\begin{tikzpicture}[scale=.2]
\draw[help lines] (0,0) grid (12,4);
\draw[thick] (0,0)--(3,3)--(4,2)--(6,4)--(8,2)--(9,3)--(12,0);
\node at (0,5)  {\color{red!90!black}\ };
\end{tikzpicture}\\  
 & $\bm{u}hh*\bm{d}$ & $\begin{matrix} (4&5&6)\\(1&4&6)\end{matrix}$ & $\begin{pmatrix}0&0&0&1&3\\1&0&0&1&0\end{pmatrix}$&
\begin{tikzpicture}[scale=.2]
\draw[help lines] (0,0) grid (12,4);
\draw[thick] (0,0)--(4,4)--(5,3)--(6,4)--(9,1)--(10,2)--(12,0);
\node at (0,4.5)  {\color{red!90!black}\ };
\end{tikzpicture}\\ \hline
  & $\bm{u}*u\bm{d}d$ & $\begin{matrix} (2&4&5&6)\\(1&2&3&6)\end{matrix}$ & $\begin{pmatrix}0&1&0&1&1\\2&1&1&0&0\end{pmatrix}$&
\begin{tikzpicture}[scale=.2]
\draw[help lines] (0,0) grid (12,3);
\draw[thick] (0,0)--(2,2)--(3,1)--(5,3)--(6,2)--(7,3)--(8,2)--(9,3)--(12,0);
\node at (0,3.5)  {\color{red!90!black}\ };
\end{tikzpicture}\\ 
$ud$ & $\bm{u}u*d\bm{d}$ & $\begin{matrix} (3&4&5&6)\\(1&2&3&6)\end{matrix}$ & $\begin{pmatrix}0&0&1&1&1\\1&1&1&0&0\end{pmatrix}$&
\begin{tikzpicture}[scale=.2]
\draw[help lines] (0,0) grid (12,3);
\draw[thick] (0,0)--(3,3)--(4,2)--(5,3)--(6,2)--(7,3)--(8,2)--(9,3)--(12,0);
\node at (0,3.5)  {\color{red!90!black}\ };
\end{tikzpicture}\\ 
 & $u\bm{u}d*\bm{d}$ & $\begin{matrix} (3&4&5&6)\\(1&2&4&6)\end{matrix}$ & $\begin{pmatrix}0&0&1&1&2\\1&1&0&1&0\end{pmatrix}$&
\begin{tikzpicture}[scale=.2]
\draw[help lines] (0,0) grid (12,3);
\draw[thick] (0,0)--(3,3)--(4,2)--(5,3)--(6,2)--(7,3)--(9,1)--(10,2)--(12,0);
\node at (0,3.5)  {\color{red!90!black}\ };
\end{tikzpicture}\\ \hline
\end{tabular}
\end{center}
\caption{The six Dyck paths of semilength 6 having $L=2$ and their corresponding Motzkin paths of length 2} \label{L2Figure}
\end{figure}

 When $L=p$ for an odd prime number, we can also enumerate $\D_n^p$ using Proposition~\ref{oneterm} as seen in the following theorem.
 
 \begin{thm}\label{TheoremLp} For a prime number $p \geq 3$ and $n \geq p$, the number of Dyck paths with semilength $n$ and $L=p$ is
\[ |\D_n^p| = 2\left(T_{n-2, p-1} + \sum_{i=0}^{n-2-p} (i+1)M_i T_{n-4-i, p-1}\right). \]
Thus, the generating function for $|\D_n^p|$ is
\[ 2x^p(1+xm(x))^{p-2}\left(x^2(xm(x))' + 1 \right). \]
\end{thm}

\begin{proof} This lemma is a direct corollary of Proposition~\ref{oneterm} with $r=1$ and $s=p-1$.  We multiply by two to account for the case where $r=p-1$ and $s=1$, and $r=1$ and $s=p-1$. \end{proof}

\section{Dyck paths with $L=4$}\label{SecL4}

When $L(D)=4$, things are more complicated than in the cases for prime numbers. If $D \in \D_n^4$, then one of the following is true:
\begin{itemize}
\item $D \in \D_n^{1,3}$ or $D \in \D_n^{3,1}$; or
\item All but two terms in the product $\prod_{i=1}^{n-1} {r_i + s_i \choose r_i}$ are equal to 1, and those terms must both equal  $2$.
\end{itemize}

Because the first case is enumerated in Section~\ref{SecRS}, this section will be devoted to counting the Dyck paths $D \in \D_n^4$ where $M^*_D$ has exactly two $*$'s in positions $k_1$ and $k_2$ and \[ L(D) = {r_{k_1} + s_{k_1} \choose r_{k_1}} {r_{k_2} + s_{k_2} \choose r_{k_2}} \] with $r_{k_1} = s_{k_1} = r_{k_2} = s_{k_2} = 1$.

For ease of notation, let $\widehat{\D}_n$ be the set of Dyck paths $D \in \D_n^4$ with the property that $M^*_D$ has exactly two $*$'s. Also, given $D \in \widehat{\D}_n$, define $x_i(D)$ to be the number of ups before the $i$th $*$ in $M^*_D$ and let $y_i(D)$ be the number of downs before the $i$th $*$ for $i \in \{1, 2\}$.

\begin{ex} Let $D $ be the Dyck path with ascent sequence and descent sequence
\[ \Asc(D) = (3, 6, 7, 8, 10, 11) \quad \text{and} \quad \Des(D) = (1, 3, 4, 5, 8, 11) \]
and thus $r$-$s$ array
\[
\left( \begin{array}{cccccccccc}
0 & 0 & 1 & 0 & 0 & 2 & 1 & 1 & 0 & 3 \\
3 & 0 & 1 & 1 & 2 & 0 & 0 & 1 & 0 & 0 \end{array} \right).\]
By inspection of the $r$-$s$ array and noticing that only columns 3 and 8 have two nonzero entries, we see that $L(D) = {1 + 1 \choose 1}{1 + 1 \choose 1} = 4$ and thus $D \in \widehat{\D}_{11}$. Furthermore, we can compute
\[ M^*_D = uh*uudd*hd.\]
Since there is one $u$ before the first $*$ and no $d$'s, we have $x_1(D) = 1$ and $y_1(D) = 0$. Similarly, there are three $u$'s before the second $*$ and two $d$'s so $x_2(D) = 3$ and $y_2(D) = 2$.
\end{ex}

In this section, we will construct $M^*$ from smaller Motzkin paths. To this end, let us notice what $M^*_{D}$ should look like if $D\in \widehat{\D}_n$. 

\begin{lem}\label{DM}
Suppose $M^*\in \M_{n-1}^*$ has exactly two $*$'s. Then $D_{M^*} \in \widehat{\D}_n$ if, writing $x_i=x_i(D_{M^*})$ and $y_i=y_i(D_{M^*})$, we have:
\begin{itemize}
\item The $(x_1+1)$st occurrence of either a $d$ or $*$ is followed by another $d$ or $*$;
\item The $(x_2+2)$nd occurrence of either a $d$ or $*$ is followed by another $d$ or $*$, or $x_2$ is equal to the number of $d$'s and $M^*$ ends in $d$ or $*$;
\item The $(y_1)$th occurrence of either a $u$ or $*$ is followed by another $u$ or $*$, or $y_1=0$ and the $M^*$ begins with $u$ or $*$;
\item The $(y_2+1)$st occurrence of either a $u$ or $*$ is followed by another $u$ or $*$.
\end{itemize}
\end{lem}

\begin{proof}
Suppose $M^*\in \M_{n-1}^*$ has two stars in positions $k_1$ and $k_2$. Then it is clear that \[L(D) ={r_{k_1} + s_{k_1} \choose r_{k_1}} {r_{k_2} + s_{k_2} \choose r_{k_2}} \] so it suffices to show that $r_{k_1} = s_{k_1} = r_{k_2} = s_{k_2} = 1$. Recall that $\Asc(D) = (a_1, a_2, \ldots, a_k)$ is the increasing sequence of positions $i$ in $M^*$ with $m_i=d$ or $*$. Similarly, $\Des(D)=(b_1, b_2, \ldots, b_k)$ is the increasing sequence of positions $i$ in $M^*$ with $m_i=u$ or $*$. 

First notice that $r_{k_1} = 1$ only if $b_i-b_{i-1}=1$ where $a_i=k_1$. However, $i=y_1+1$ since the first star must be the $(y_1+1)$st occurrence of $d$ or $*$. Therefore $b_i$ is the position of the $(y_1+1)$st $u$ or $*$ and $b_{i-1}$ is the position of the $(y_i)$th $u$ or $*$. The difference in positions is 1 exactly when they are consecutive in $M^*$. 
The other three bullet points follow similarly.
\end{proof}

Enumerating Dyck paths $D \in \widehat{\D}_n$ will be found based on the values of $x_1(D)$ and $y_2(D)$. The cases we consider are
\begin{itemize}
\item $x_1(D) \notin \{y_2(D), y_2(D) + 1\}$;
\item $x_1(D) = 1$ and $y_2(D) = 0$;
\item $x_1(D) = y_2(D) + 1 \geq 2$; and
\item $x_1(D) = y_2(D)$.
\end{itemize}

The next four lemmas address each of these cases separately. Each lemma is followed by an example showing the correspondence the proof provides.

\begin{lem} \label{L4Type1} For $n \geq 7$, the number of Dyck paths $D \in \widehat{\D}_n$ with  $x_1(D) \notin \{y_2(D), y_2(D) + 1\}$ is ${n-5 \choose 2}M_{n-7}.$
\end{lem}

\begin{proof} 

We will show that for any $M \in \M_{n-7}$, there are ${n-5 \choose 2}$ corresponding Dyck paths $D \in \widehat{\D}_n$ with  $x_1(D) \notin \{y_2(D), y_2(D) + 1\}$. To this end, let $M \in \M_{n-7}$ and let $1 \leq j_1< j_2 \leq n-5$.  There are ${n-5 \choose 2}$ choices for $j_1$ and $j_2$ each corresponding to a Dyck path with the desired properties. We create a modified Motzkin word $\ol{M}^* \in \M_{n-5}^*$ with $*$'s in position $j_1$ and $j_2$ and the subword of $\ol{M}^*$ with the $*$'s removed is equal to $M$.
Let $\overline{x}_i$ be the number of ups before the $i$th $*$ in $\ol{M}^*$ and let $\overline{y}_i$ be the number of downs before the $i$th $*$ in $\ol{M}^*$ for $i \in \{1, 2\}$. We create the modified Motzkin word $M^* \in M^*_{n-1}$ from $\ol{M}^*$ as follows:
\begin{enumerate}
\item Insert $d$ before the $(\overline{x}_2+1)$th down or at the very end of $\ol{x}_2$ is the number of downs in $\ol{M}^*$. 
\item Insert $d$ before the $(\overline{x}_1+1)$th down or at the very end of $\ol{x}_1$ is the number of downs in $\ol{M}^*$.
\item Insert $u$ after the $\overline{y}_2$th up or at the beginning if $\overline{y}_2 = 0$.
\item Insert $u$ after the $\overline{y}_1$th up or at the beginning if $\overline{y}_1 = 0$. 
\end{enumerate}
Notice that in Step (1), the $d$ is inserted after the second $*$ and in Step~(4), the $u$ is inserted before the first $*$.  Let $D=D_{M^*}$. We first show that $D \in  \widehat{\D}_n$ by showing $L(D)=4$.  We proceed by examining two cases.

In the first case, assume $\ol{x}_1 +1 \leq \ol{y}_2.$ In this case, the inserted $d$ in Step~$(2)$ must occur before the second $*$ since there were $\ol{y}_2$ $d$'s before the second $*$ in $M^*$.  Similarly, the inserted $u$ in Step~$(3)$ must occur after the first $*$.   Thus, we have
\[ x_1(D) = \ol{x}_1 + 1, \quad y_1(D) = \ol{y_1}, \quad x_2(D) = \ol{x}_2 + 2, \quad \text{and} \quad  y_2(D) = \ol{y}_2 + 1.\] We now use the criteria of Lemma~\ref{DM}, to see that $L(D) = 4$:
\begin{itemize}
\item The $(x_1 + 1)$th occurrence of a $d$ or $*$ is the inserted $d$ from Step (2) and is thus followed by $d$;
\item The $(x_2 + 2)$th occurrence of a $d$ or $*$ is the inserted $d$ from Step (1) and is thus followed by $d$;
\item The $y_1$th occurrence of a $u$ is the inserted $u$ from Step (4) and is thus followed by $u$; and
\item The $(y_2 + 1)$th occurrence of a $u$ is the inserted $u$ from Step (3) and is thus followed by $u$.
\end{itemize}

We also have
\[ x_1(D) = \ol{x}_1 + 1 \leq \ol{y}_2 < \ol{y_2} + 1 = y_2(D), \]
and thus $D \in \widehat{\D}_n$ with  $x_1(D) \notin \{y_2(D), y_2(D) + 1\}$ as desired.

In the second case where $\ol{x}_1  \geq  \ol{y}_2$, the inserted $d$ in Step (2) occurs after the second $*$ and the inserted $u$ in Step (3) occurs before the first $*$. Here we have
\[ x_1(D) = \ol{x}_1 + 2, \quad y_1(D) = \ol{y_1}, \quad x_2(D) = \ol{x}_2 + 2, \quad \text{and} \quad  y_2(D) = \ol{y}_2.\]
We can easily check that the criteria of Lemma~\ref{DM} are satisfied to show that $L(D) = 4$. Also,
\[ x_1(D) = \ol{x}_1 + 2 \geq \ol{y}_2 + 2 = y_2(D) + 2,\] and thus $D$ has the desired properties.

\sloppypar{To see that this process is invertible, consider any $D \in \widehat{\D}_n$ with  $x_1(D) \notin \{y_2(D), y_2(D) + 1\}$ and let $k_1 \leq k_2$ be the positions of the $*$'s in $M^*_D$. We consider the two cases where $x_1(D) < y_2(D)$ and where $x_1(D) \geq y_2(D)+2$. Since for each case, we've established the relationship between $x_i$ and $\overline{x}_i$ and between $y_i$ and $\overline{y}_i$, it is straightforward to undo the process.  }

Begin with the case where $x_1(D) < y_2(D)$. In this case:
\begin{itemize}
\item Delete the $(x_2(D))$th $d$ and the $(x_1(D))$th $d$.
\item Delete the $(y_2(D) + 1)$th $u$ and the $(y_1(D)+1)$th $u$.
\item Delete both $*$'s.
\end{itemize}
Now consider the case where $x_1(D) \geq y_2(D)+2$. In this case:
\begin{itemize}
\item Delete the $(x_2(D)+2)$th $d$ and the $(x_1(D)+1)$th $d$.
\item Delete the $(y_2(D) + 1)$th $u$. and the $(y_1(D)+2)$th $u$.
\item Delete both $*$'s.
\end{itemize}
 \end{proof}

 \begin{ex} Suppose $n=11$ and let $M = hudh \in \M_{4}$. There are ${6 \choose 2} = 15$ corresponding Dyck paths $D \in \widehat{\D}_n$ with  $x_1(D) \notin \{y_2(D), y_2(D) + 1\}$, and we provide two of these in this example.
 
 First, suppose $j_1 = 2$ and $j_2=5$ so that $\ol{M}^* = h*ud*h$. We then count the number of ups and downs before each $*$ to get
 \[ \ol{x}_1 = 0, \quad \ol{y}_1 = 0, \quad \ol{x}_2 = 1, \quad \text{and} \quad \ol{y}_2 = 1.\]
 Following the steps in the proof, we insert two $u$'s and two $d$'s to get
 \[ M^* = \bm{u}h*u\bm{ud}d*h\bm{d}.\] Let $D=D_{M^*}$ and notice that the number of ups before the first $*$ is $x_1(D) = 1$ and the number of downs before the second $*$ is $y_2(D) = 2$ and thus $x_1(D) < y_2(D)$. Since $L(D) = 4$, $D$ satisfies the desired criteria. To see that the process is invertible, we would delete the third and first $d$, the third and first $u$, and the two $*$'s.
 
 Now, suppose $j_1=2$ and $j_2=4$ so that $\ol{M}^* = h*u*dh$. We again count the number of ups and downs before each $*$ to get
 \[ \ol{x}_1 = 0, \quad \ol{y}_1 = 0, \quad \ol{x}_2 = 1, \quad \text{and} \quad \ol{y}_2 = 0,\] and insert two $u$'s and two $d$'s to get
 \[ M^* = \bm{uu}h*u*\bm{d}dh\bm{d}. \]
 Now, if $D=D_{M^*}$ we have $x_1(D)= 2$ and $y_2(D) = 0$ and so $x_1(D) \geq y_2(D) + 2$ as desired.  We can also easily check the $L(D) = 4$.
 \end{ex}
 
\begin{lem}\label{L4Type2} For $n \geq 5$, the number of Dyck paths $D \in \widehat{\D}_n$ with  $x_1(D)= 1$ and $y_2(D)= 0$ is $M_{n-5}.$
\end{lem}

\begin{proof} We find a bijection between the set of Dyck paths $D \in \widehat{\D}_n$ with  $x_1(D)= 1$ and $y_2(D)= 0$ with the set $\M_{n-5}$. First, suppose $M \in \M_{n-5}$. Let $\ol{x}_2$ be the number of ups before the first down.  Now create the modified Motzkin word $M^* \in \M^*_{n-1}$ as follows.
\begin{enumerate}
\item Insert $d$ before the $(\ol{x}_2 + 1)$st $d$ in $M$ or at the end if the number of downs in $M$ is $\ol{x}_2$.
\item Insert $*$ before the first $d$.
\item Insert $u$ followed by $*$ before the first entry.
\end{enumerate}
Let $D = D_{M^*}$. By construction, we have
\[ x_1(D) = 1, \quad y_1(D) = 0, \quad x_2(D) = \ol{x} + 1, \quad \text{and} \quad y_2(D) =0.\]
In particular, Step (3) gives us $x_1(D)$ and $y_1(D)$, while Step (2) gives us $x_2(D)$, and $y_2(D)$. We also have the four criteria of Lemma~\ref{DM}:
\begin{itemize}
\item The second occurrence of a $d$ or $*$ is the second $*$ which is followed by $d$;
\item The $(x_2+2)$st occurrence of a $d$ or $*$ is the inserted $d$ from Step (1) which is followed by $d$ (or at the end);
\item $M^*$ begins with $u$; and
\item The first occurrence of a $u$ or $*$ is the first entry $u$ which is followed by $*$.
\end{itemize}

Let us now invert the process. Starting with a Dyck path $D$ with $x_1(D)= 1$ and $y_2(D)= 0$, and its corresponding modified Motzkin word $M^*_D$. Since $y_2(D)=0$, we also have $y_1(D)=0$, and thus by Lemma~\ref{DM}, we must have that the first two entries are either $uu$, $u*$, $**$, or $*u$. However, since we know $x_1(D)=1$, it must be the case that $M^*_D$ starts with $u*$. As usual, let $x_2(D)$ be the number of $u$'s before the second star. Obtain $M\in M_{n-5}$ by starting with $M^*_{D}$ and then:
\begin{itemize}
\item Delete the $(x_2)$th $d$.
\item Delete the first $u$.
\item Delete both $*$'s.
\end{itemize}
\end{proof}

\begin{ex} Suppose $n=11$ and let $M=huudhd \in \M_6$. By Lemma~\ref{L4Type2}, there is one corresponding Dyck path $D \in \widehat{\D}_{11}$ with  $x_1(D)= 1$ and $y_2(D)= 0$. Following the notation in the proof, we have $\ol{x}_2 = 2$ and we get
\[ M^* = \bm{u*}huu\bm{*d}hd\bm{d}.\]
Let $D=D_{M^*}$. We can easily check that $L(D)$ = 1. Also, $x_1(D) = 1$ and $y_2(D) = 0$ as desired.
\end{ex}

\begin{lem} \label{L4Type3} For $n \geq 7$, the number of Dyck paths $D \in \widehat{\D}_n$ with  $x_1(D)=  y_2(D)+1 \geq 2$ is $$\sum_{i=0}^{n-7} (i+1)M_iM_{n-7-i}.$$
\end{lem}

\begin{proof} Consider a pair of Motzkin paths, $M$ and $P$, where $M \in \M_i$ and $P \in \M_{n-7-i}$ with $0 \leq i \le n-7$. For each such pair, we consider $1 \leq j \leq i+1$ and find a corresponding Dyck path $D \in \widehat{\D}_n$ with  $x_1(D)=  y_2(D)+1 \geq 2$. Thus, there will be $i+1$ corresponding Dyck paths for each pair $M$ and $P$.

We begin by creating a modified Motzkin word $\ol{M}^* \in \M^*_{n-4}$ by inserting $u$, followed by $*$, followed by the path $P$, followed by $d$ before the $i$th entry in $M$. Then, let $\overline{x}_1$ be the number of ups before $*$ in $\ol{M}^*$ and $\overline{y}_1$  be the number of downs before the $*$ in $\ol{M}^*$. Notice that $\ol{x}_1 \geq 1$. We now create a modified Motzkin word $M^* \in \M^*$ as follows.
\begin{enumerate}
\item Insert $*$ before the $(\ol{x}_1 + 1)$st $d$ or at the end if $\ol{x} + 1$ equals the number of downs in $\ol{M}^*$. Now let  $\overline{x}_2$ be the number of ups before this second $*$. 
\item Insert $u$ after the $\ol{y}_1$th $u$ in $\ol{M}^*$ or at the beginning if $\ol{y}_1 = 0$.
\item Insert $d$ before the $(\ol{x}_2 + 1)$st $d$ (or at the end).
\end{enumerate}
Let $D = D_{M^*}$. By construction, we have
\[ x_1(D) = \ol{x}_1 + 2, \quad y_1(D) = \ol{y}_1, \quad x_2(D) = \ol{x}_2 + 1, \quad \text{and} \quad y_2(D) = \ol{x}_1,\]
and thus $x_1(D)=  y_2(D)+1 \geq 2$.

We also have the four criteria of Lemma~\ref{DM}:
\begin{itemize}
\item The $(x_1 + 1)$st occurrence of a $d$ or $*$ is the second $*$ from Step (1) which is followed by $d$;
\item The $(x_2+2)$st occurrence of a $d$ or $*$ is the inserted $d$ from Step (3) which is followed by $d$ (or at the end);
\item The $(y_1+1)$st occurrence of a $u$ or $*$ is the inserted $u$ from Step (2) and thus is preceded by $u$; and
\item The $(y_2 + 2)$nd occurrence of a $u$ or $*$ is the first $*$ which immediately follows a $u$.
\end{itemize}

To see that this process is invertible, consider any Dyck path $D \in \widehat{\D}_n$ with  $x_1(D)=  y_2(D)+1 \geq 2$. To create $\ol{M}^*$, start with $M^*_D$ and then:
\begin{itemize}
\item Delete the $(x_2)$th $d$.
\item Delete the second $*$.
\item Delete the $(y_1 + 1)$st $u$.
\end{itemize}
Because $x_1 = y_2 + 1$, we have $y_1 + 1 \leq  x_2$ and so this process results in $\ol{M}^* \in \M^*_{n-4}$. Now let $P$ be the maximal subpath in $\ol{M}^*$ beginning with the entry immediately following the $*$. Deleting the $u$ and the $*$ preceding $P$, all of $P$, and the $d$ following $P$ inverts the process.
\end{proof}

\begin{ex} Suppose $n=11$ and let $M = ud \in \M_{2}$ and $P = hh \in \M_2$. There are 3 corresponding Dyck paths with $D \in \widehat{\D}_n$ with  $x_1(D)=  y_2(D)+1 \geq 2$ and we provide one example. First, let $j=1$ and create the word $\ol{M}^*\in \M^*_{7}$ by inserting $u*Pd$ before the first entry in $M$:
\[ \ol{M}^* = \framebox{$u*hhd$}\ ud.\]Notice $\ol{x}_1 = 1$ and $\ol{y}_1=0$ since there is only one entry, $u$, before the $*$.  Then, following the procedure in the proof of Lemma~\ref{L4Type3}, we insert $*$ before the second $d$ and note that $\ol{x}_2 = 2.$ Then we insert $u$ at the beginning and $d$ at the end to get
\[ M^* = \bm{u}  \framebox{$u*hhd$}\ u\bm{*}d\bm{d}. \] Let $D=D_{M^*}$. By inspection, we note $x_1(D) = 2$ and $y_2(D) = 1$, and we can easily check that $L(D) = 4$.
\end{ex}

\begin{lem} \label{L4Type4} For $n \geq 7$, the number of Dyck paths $D \in \widehat{\D}_n$ with  $x_1(D)=  y_2(D)$ is $$\sum_{i=0}^{n-7} (i+1)M_iM_{n-7-i}.$$ Also, for $n=3$, there is exactly 1 Dyck path $D \in \widehat{\D}_3$ with $x_1(D)=  y_2(D)$.
\end{lem}

\begin{proof}
Similar to the proof of Lemma~\ref{L4Type3}, consider a pair of Motzkin paths, $M$ and $P$, where $M \in \M_i$ and $P \in \M_{n-7-i}$ with $0 \leq i \le n-7$. For each such pair, we consider $1 \leq j \leq i+1$ and find a corresponding Dyck path $D \in \widehat{\D}_n$ with  $x_1(D)=  y_2(D)$. Thus, there will be $i+1$ corresponding Dyck paths for each pair $M$ and $P$.

We begin by creating a modified Motzkin word $\ol{M}^* \in \M^*_{n-4}$ by inserting $*$, followed by $u$, followed by the path $P$, followed by $d$ before the $j$th entry in $M$. Then, let $\overline{x}_1$ be the number of ups before $*$ in $\ol{M}^*$ and $\overline{y}_1$  be the number of downs before the $*$ in $\ol{M}^*$.  We now create a modified Motzkin word $M^* \in \M^*$ as follows.
\begin{enumerate}
\item Insert $*$ after the $(\ol{x}_1 + 1)$st $d$  in $\ol{M}^*$. Let  $\overline{x}_2$ be the number of ups before this second $*$. 
\item Insert $u$ after the $\ol{y}_1$th $u$ in $\ol{M}^*$ or at the beginning if $\ol{y}_1 = 0$.
\item Insert $d$ before the $(\ol{x}_2 + 1)$st $d$ (or at the end).
\end{enumerate}
Let $D = D_{M^*}$. By construction, we have
\[ x_1(D) = \ol{x}_1 + 1, \quad y_1(D) = \ol{y}_1, \quad x_2(D) = \ol{x}_2 + 1, \quad \text{and} \quad y_2(D) = \ol{x}_1+1.\] It is easy to verity that the criteria in Lemma~\ref{DM} are satisfied and so $D \in \widehat{\D}_n$ with  $x_1(D)=  y_2(D)$.

To see that this process is invertible, consider any Dyck path $D \in \widehat{\D}_n$ with  $x_1(D)=  y_2(D)$. Since $x_1=y_2$, there are $y_2$ ups before the first $*$, in $M^*_D$ and thus the first $*$ in $M^*_D$ is the $(y_2+1)$th occurrence of a $u$ or $*$.  By the fourth criterium in Lemma~\ref{DM}, the first $*$ must be followed by another $u$ or $*$. Similarly, the $(x_1 + 2)$th occurrence of either a $d$ or $*$ is the second $*$.  Thus, by the second criterium of Lemma~\ref{DM}, the second $*$ must be immediately preceded by a $d$ or $*$.

We now show that the case where the first $*$ is immediately followed by the second $*$ results in only one Dyck path.  In this case, $x_1=y_1=x_2=y_2$, and thus $M^*_D$ can be decomposed as a Motzkin path, followed by $**$, followed by another Motzkin path.  By the second criterium in Lemma~\ref{DM}, the entry after the second $*$ must be a $d$ (which is not allowed) and thus $M^*_D$ ends in $*$. Similarly, the third criterium in Lemma~\ref{DM} tells us $M^*_D$ begins with $*$ and so $M^*_D = **$.  Thus, $D \in \widehat{\D}_n$ and is the path $D=ududud$.

We now assume the first $*$ is followed by $u$ which implies the second $*$ is preceded by $d$. In this case, we must have at least $y_1 + 1$ downs before the second $*$ and at least $x_1 + 1$ ups before the second $*$ yielding
\[ y_1 + 1 \leq x_1 = y_2 \leq x_2 - 1.\]
Thus the $(y_1+1)$th $u$ comes before the first $*$ and the $x_2$th $d$ comes after the second $*$. To find $\ol{M}^*$ from $M^*_D$:
\begin{itemize}
\item Delete the $x_2$th $d$.
\item Delete the second $*$.
\item Delete the $(y_1 + 1)$st $u$;
\end{itemize}
which results in  $\ol{M}^* \in \M^*_{n-4}$. Now let $P$ be the maximal subpath in $\ol{M}^*$ beginning with the entry after the $u$ that immediately follows the remaining $*$. (The entry after $P$ must be $d$ since $P$ is maximal and  $\ol{M}^*$ is a Motzkin path when ignoring the $*$.) Removing $uPd$ and the remaining $*$ from $\ol{M^*}$ results in a Motzkin path $M$ as desired.
\end{proof}

\begin{ex} Suppose $n=11$ and let $M = ud \in \M_{2}$ and $P = hh \in \M_2$. There are 3 corresponding Dyck paths with $D \in \widehat{\D}_n$ with  $x_1(D)=  y_2(D)$ and we provide one example. First, let $j=1$ and create the word $\ol{M}^*\in \M^*_{7}$ by inserting $*uPd$ before the first entry in $M$:
\[ \ol{M}^* = \framebox{$*uhhd$}\ ud.\]Notice $\ol{x}_1 = 0$ and $\ol{y}_1=0$ since there are no entries before the $*$.  Then, following the procedure in the proof of Lemma~\ref{L4Type4}, we insert $*$ after the first $d$ and note that $\ol{x}_2 = 1.$ Then we insert $u$ at the beginning and $d$ before the second $d$ to get
\[ M^* = \bm{u}  \framebox{$*uhhd$}\bm{*}u\bm{d}d. \] Let $D=D_{M^*}$. By inspection, we note $x_1(D) = y_2(D) = 1$, and we can easily check that $L(D) = 4$.
\end{ex}

\begin{thm}\label{thm:L4}
The number of Dyck paths with semilength $n \geq 4$ and $L=4$ is
\[ |\D_n^4| =2\left(T_{n-2, 3} + \sum_{i=0}^{n-6} (i+1)M_i T_{n-4-i, 3}\right)  + \binom{n-5}{2}M_{n-7} + M_{n-5} + 2\sum_{i=0}^{n-7} (i+1)M_i M_{n-7-i}.  \] Also, $|\D_3^4| = 1$.
\end{thm}

\begin{proof} This is a direct consequence of Proposition~\ref{oneterm} along with Lemmas~\ref{L4Type1}, \ref{L4Type2}, \ref{L4Type3}, and \ref{L4Type4}. \end{proof}

\section{Further Remarks}\label{SecRemarks}
As seen in Section~\ref{SecL4}, finding $|\D_n^k|$ is more complicated when $k$ is not prime, as there could be many ways to write $k$ as a product of binomial coefficients.  For example, consider $k=6$. If $D \in \D_n^6$, then one of the following is true:
\begin{itemize}
\item $D \in \D_n^{1,5}$ or $D \in \D_n^{5,1}$;
\item $D \in \D_n^{2,2}$; or
\item All but two terms in the product $\prod_{i=1}^{n-1} {r_i + s_i \choose r_i}$ are equal to 1, and those terms must equal  $2$ and $3$.
\end{itemize}

The number of Dyck paths in the first two cases is given by Proposition~\ref{oneterm}:
\[ |\D_n^{1,5}| = |\D_n^{5,1}| = T_{n-2,5} + \sum_{i=0}^{n-8}(i+1)M_i T_{n-4-i, 5}\]
and
\[  |\D_n^{2,2}|  = T_{n-2,3} + \sum_{i=0}^{n-6}(i+1)M_i T_{n-4-i, 3}.\]

In the final case, we have
\[ L(D) = {r_{k_1} + s_{k_1} \choose r_{k_1}} {r_{k_2} + s_{k_2} \choose r_{k_2}} \] where exactly one of $\{r_{k_1}, r_{k_2}, s_{k_1},s_{k_2}\}$ is equal to 2 and the other three values are 1. By symmetry, we need only to consider two cases: when $r_{k_1} = 2$ and when $s_{k_1} = 2$.  We can appreciate that these cases can become quite involved; the proofs would involve similar techniques to those found in Section~\ref{SecL4} along with the proof of Proposition~\ref{oneterm}. Although we do not provide a closed form, the number of Dyck paths $D \in \D_n^6$ in this case are (starting at $n=4$):
\[ 2, 4, 8, 16, 44, 122, 352, 1028, 3036, \ldots.\]
Combining the first two cases with this case, we provide the first terms of the values of $|\D_n^6|$ (starting at $n=4$):
\[ 3, 6, 14, 34, 92, 252, 710, 2026, 5844, \ldots.\]

Further work in this area could involve finding formulas for $|\D_n^k|$ when $k$ is a non-prime number greater than 4. It also still remains open to refine the enumeration of $\D_n^k$ with respect to the number of returns. Having such a refinement in terms of number of returns would yield a new formula for the number of 321-avoiding permutations of length $3n$ composed only of 3-cycles as seen in Equation~(\ref{eqnSumL2}).

\bibliographystyle{amsplain}

\begin{thebibliography}{99}

\bibitem{ArcGra21} K. Archer and C. Graves, Pattern-restricted permutations composed of 3-cycles, \emph{Discrete Mathematics} \textbf{345 (7)} (2022) doi: 10.1016/j.disc.2022.112895.

\bibitem{Aigner98} M. Aigner, Motzkin numbers, \emph{European Journal of Combinatorics} \textbf{19} (1998) 663-675.

\bibitem{Brualdi} R. A. Brualdi, \emph{Introductory Combinatorics, 4th ed.} New York: Elsevier (1997).

\bibitem{GuProd20} N. S. S. Gu and H. Prodinger, A bijection between two subfamilies of Motzkin paths, \emph{Applicable Analysis and Discrete Mathematics}, \textbf{15(2)}, (2021), 460--466.

\bibitem{Prod} H. Prodinger, An elementary approach to solve recursions relative to the enumeration of S-Motzkin paths. \emph{Journal of Difference Equations and Applications}, \textbf{27(5)}, (2021) 776-785.


\bibitem{Sulanke} R. A. Sulanke, Generalizing Narayana and Schroeder Numbers to Higher Dimensions, \emph{Electronic Journal of Combinatorics} \textbf{11} (2004), Research Paper 54, 20 pp.

\bibitem{Zeil} D. Zeilberger, Andre's reflection proof generalized to the many-candidate ballot problem, \emph{Discrete Mathematics} \textbf{44(3)} (1983), 325--326.

\end{thebibliography}

\end{document}